\documentclass[11pt]{amsart}

\usepackage{amssymb}
\usepackage{amsmath}
\usepackage{epsfig}
\usepackage{color}
\usepackage{graphics}

\newcommand{\RR}{{\mathbb R}}

\newcommand{\NN}{{\mathbb N}}
\newcommand{\ZZ}{\mathbb Z}

\def\A{{\mathcal A}}

\def\D{{\mathcal D}}

\def\F{{\mathcal F}}

\def\R{{\mathcal R}}

\def\diam{{\rm diam}}
\def\vol{{\rm vol}}

\numberwithin{equation}{section}

\newtheorem{theo}{Theorem}
\newtheorem{prop}[theo]{Proposition}

\newtheorem{lemma}[theo]{Lemma}

\newtheorem{remark}[theo]{Remark}

\begin{document}

\parindent = 0cm

\title
{Linearly repetitive Delone systems have a finite number of non
periodic Delone system  factors.}

\author{Mar\'{\i}a Isabel Cortez}
\address{Departamento de Matem\'atica y CC. de la Universidad de Santiago de Chile,
Av. Libertador Bernardo O'Higgins 3363.} \email{mcortez@usach.cl }

\author{Fabien Durand}
\address{Laboratoire Ami\'enois
de Math\'ematique Fondamentale et Appliqu\'ee, CNRS-UMR 6140,
Universit\'{e} de Picardie Jules Verne, 33 rue Saint Leu, 80039
Amiens Cedex, France.} \email{fabien.durand@u-picardie.fr}
\email{samuel.petite@u-picardie.fr}
\author{Samuel Petite}

\subjclass{37B50, } \keywords{Delone sets, tiling systems, factor
maps, linearly repetitive, Vorono\"{\i} cell.}

\begin{abstract}
We prove  linearly repetitive Delone systems have finitely many
Delone system factors up to conjugacy. This result is also
applicable to linearly repetitive tiling systems.
\end{abstract}

\maketitle \markboth{Maria Isabel Cortez, Fabien Durand, Samuel
Petite}{ }

\section{Introduction}

The concepts of tiling dynamical  system and Delone dynamical
system are extensions to $\RR^d$-actions of the notion of subshift
(see \cite{Ro}). Classical examples  are those generated by
self-similar tilings, as the Penrose one, which have been
extensively studied since the 90's.  For details and references
see for example \cite{Ro,So1}. Systems arising from self-similar
tilings are known to be linearly repetitive, this means there
exists a positive constant $L$, such that every pattern of
diameter $D$ appears in every ball of radius $LD$ in any tiling of
the system. This concept has been first defined in \cite{LP}.
Linearly repetitive tiling and Delone systems can be seen as a
generalization to $\RR^d$-actions of the notion of linearly
recurrent subshift introduced in \cite{DHS}.

We study  the factor maps between Delone systems. The main result
is the following: linearly repetitive Delone systems have finitely
many Delone system factors up to conjugacy. As noticed in
\cite{So3}, tiling systems are topologically conjugate to Delone
systems. This conjugacy also preserves linear repetitivity.
Consequently, the results that we present can be easily extended
to linearly repetitive tiling systems.

The main result of this paper was obtained in the context of
subshifts in \cite{Du1}. A key tool used in \cite{Du1}, is the
existence of sliding-block-codes for factor maps between subshifts
(Curtis-Hedlund-Lyndon Theorem). Unlike subshifts, factor maps
between two tiling systems are not always sliding-block-codes (see
\cite{Pe} and \cite{RS}). The lack of this property appears to be
the main difficulty of this work. To surmount this obstacle, we
carefully dissect continuity of factor maps, by means of
Vorono\"{\i} cells and return vectors.

This paper is organized as follows: in Section
\ref{Definitions_and_background} we recall basic concepts and
results about Delone systems. In Section
\ref{Preimages_of_factor_maps}  we show the factor maps from
linearly repetitive Delone systems to Delone systems  are
finite-to-one. Finally, Section
\ref{Number_of_factors_of_linearly_repetitive_Delone_systems} is
devoted to the proof of the main theorem.


\section{Definitions and background}
\label{Definitions_and_background}
\label{definitions} In  this section we give the basic definitions
and properties concerning Delone sets. For more details we refer
to \cite{LP} and \cite{Ro}. Let $r$ and $R$ be two positive real
numbers. A $(r,R)$-{\em Delone set} $X$ is a discrete subset of
$\RR^d$ satisfying the following two properties:

\begin{enumerate}
 \item
{\em Uniform discreteness}: each open ball of radius $r >0$ in
$\RR^d$ contains at most one point of $X$. \item {\em Relative
density}: each closed ball of radius $R>0$ in $\RR^d$ contains at
least one point of $X$.
\end{enumerate}
A
$(r,R)$-Delone set $X$, in short a {\em Delone set},  is  of {\em
finite type} if $X - X$ is {\em locally finite}, i.e. the
intersection of $X -X$ with any bounded set is finite. \\
The translation by a vector $v \in \RR^d$ of a Delone set $X$, is
the Delone set  $X-v$ obtained after translating every point of
$X$ by $-v$.  Observe that  $X-v$ is of finite type if and only if
$X$ is of finite type.  A Delone set is said to be {\em non
periodic} if $X-v =X$ implies $v=0$.

Let $R>0$ and $X$ be a Delone set. We say that $P\subseteq X$ is
the $R$-{\em patch} of $X$ centered at the point $y\in \RR^d$ if
$$P = X \cap B_R (y),$$
where $B_R (y)$ denotes the open ball of a radius $R$ centered at
$y$. If there is no confusion, we refer to a $R$-patch of $X$
merely as a patch. A {\em sub-patch} of the patch $P$  is a
patch of $X$ included in $P$. 
 A patch $Q$ is a {\em translated} of the patch  $P$  if there
exists $v\in \RR^d$ such that $P-v=Q$. The vector  $v\in \RR^d$ is
a {\it return vector} of  the patch $P$ in $X$ if $P-v$ is a patch
of $X$. An {\em occurrence} of the patch $P$ of $X$ centered at
$y\in \RR^d$ is a point $w\in \RR^d$  such that $y-w$ is a return
vector of $P$. Observe the patch $P-(y-w)$ is the  translated of
$P$ centered at $w$.

The $R$-{\em atlas} $\A_X(R)$ of $X$ is the collection of all the
$R$-patches centered at a point  of $X$ translated to the origin.
More precisely:
$$ \A_{X}(R) = \{ X\cap B_R(x) -x ; \  x\in X \}.
$$
The atlas $\A_X$ of $X$ is the union of all the $R$-atlases, for
$R> 0$. Notice that $X$ is of  finite type if and only if
$\A_{X}(R)$ is finite for every $R>0$.

The Delone set $X$ is {\em repetitive} if for each $R > 0$ there
is a finite number $M > 0$, such that for every closed ball $B$ of
radius $M$ the set $B\cap X$ contains a translated patch of every
$R$-patch of $X$. Observe that any repetitive Delone set is
necessarily of finite type.

The {\em Vorono\"{\i} cell} of a point $x \in X$ is the compact
subset
$$V_{x} = \{ y \in \RR^d; \  \vert\vert x-y \vert\vert \leq \vert\vert x'-y \vert\vert \ \mbox{ for any} \ x' \in X\}. $$
Notice that if $X$ is a Delone set of finite type, then each
Vorono\"{\i} cell of $X$ is a polyhedra, and there is a finite
number of Vorono\"{\i} cells of $X$ up to translations.

\subsection{Delone systems}
We denote by $\D$ the collection of the Delone sets of $\RR^d$.
The group $\RR^d$ acts on $\D$ by  translations:
$$(v,X)\mapsto X-v  \mbox{ for } v\in \RR^d \mbox{ and } X\in \D.$$
Furthermore, this action is continuous with the topology induced
by the following distance: take $X$, $X'$ in $\D$, and define $A$
the set of $\varepsilon \in (0,\frac{1}{\sqrt{2}})$ such that
there exist $v$ and $v'$ in $B_{\varepsilon}(0)$  with
$$(X-v)\cap B_{1/\varepsilon}(0)=(X'-v')\cap
B_{1/\varepsilon}(0),$$ we set
$$
d(X,X')=\left\{\begin{array}{cc}
                  \inf A & \mbox{ if } A\neq \emptyset \\
                 \frac{1}{\sqrt{2}} & \mbox{ if } A=\emptyset.\\
               \end{array}\right.
$$
Roughly speaking, two Delone sets are close if they have the same
pattern in a large neighborhood of the origin, up to a small
translation.


A {\em Delone system} is a pair $(\Omega, \RR^d)$ such that
$\Omega$ is a translation invariant closed subset of $\D$. The
orbit closure of a Delone set $X$ in $\D$ is the set
$\Omega_{X}=\overline{\{X+v: v\in \RR^d\}}$. This is invariant by
the $\RR^d$-action, and, it is compact if and only if $X$ is of
finite type (see \cite{Ro} and \cite{Ru}). Every $X' \in
\Omega_{X}$ is a $(r,R)$-Delone set if $X$ is a $(r,R)$-Delone
set, and for any real $R>0$, we have $\A_{X'} (R)\subset
\A_{X}(R)$. If all the orbits are dense in $\Omega_{X}$, the
Delone system $(\Omega_{X},\RR^d)$ is said to be {\em minimal}. It
is shown in \cite{Ro} that the Delone set $X$ is repetitive if and
only if the system $(\Omega_{X},\RR^d)$ is minimal. In that case,
for any $X' \in \Omega_{X}$ and any $R>0$ the $R$-atlases $\A_{X'}
(R), \A_{X}(R)$ are the same. If in addition, $X$ is non periodic,
then every Delone set in $\Omega_{X}$ is non periodic.




A {\em factor map}   between two Delone systems
$(\Omega_{1},\RR^d)$ and $(\Omega_2,\RR^d)$ is a continuous
surjective map $\pi:\Omega_1\to \Omega_2$ such that
$\pi(X-v)=\pi(X)-v$, for every $X\in \Omega_1$ and $v\in \RR^d$.

In symbolic dynamics it is well-known that topological factor maps
between subshifts are always given by sliding-block-codes.  There
are examples  which show that this result can not be extended to
Delone systems (\cite{Pe}, \cite{RS}). The following lemma shows
that factor maps between Delone systems are not far to be
sliding-block-codes. A similar result  can be found in \cite{HRS}.

\begin{lemma}
\label{semi-sliding-block-code} Let $X_1$ and $X_2$ be two Delone
sets.
Suppose $X_1$ is of finite type and  $\pi: \Omega_{X_1}\to
\Omega_{X_2}$ is a factor map. Then, there exists a constant
$s_0>0$ such that for every   $\varepsilon>0$, there exists
$R_{\varepsilon}>0$ satisfying the following: For any  $R\geq
R_{\varepsilon}$, if  $X$ and $X'$  in $\Omega_{X_1}$ verify
$$X\cap B_{R+s_0}(0)=X'\cap B_{R+s_0}(0),$$  then
$$(\pi(X)-v)\cap B_{R}(0)=\pi(X')\cap B_{R}(0)$$ for some $v\in
B_{\varepsilon}(0)$.
\end{lemma}

\begin{proof}
The Delone set $X_2$ is also of finite type because $\Omega_{X_2}$
is compact. Let $r_0$ and $R_0$ be a positive constant such that
$X_2$ is a $(r_0,R_0)$-Delone set. Since all the elements of
$\Omega_{X_{2}}$ are $(r_0, R_0)$-Delone sets, if two different
points $y_1,y_2$ of $\RR^d$ satisfy $(X-y_1)\cap
B_R(a)=(X-y_2)\cap B_R(a)$ for some $X\in \Omega_{X_2}$, $a \in
\RR^d$  and  $R> R_{0}$, then $\|y_1-y_2\|\geq \frac{r_{0}}{2}$
(for the details see \cite{So1}).

\medskip

Let  $0<\delta_0<\min\{\frac{r_{0}}{4}, \frac{1}{ R_{0}} \}$.
Since $\pi$ is uniformly continuous, there exists $s_0>1$  such
that if $X$ and $X'$  in $\Omega_{X_1}$ verify $X\cap
B_{s_0}(0)=X'\cap B_{s_0}(0)$ then
  $$(\pi(X)-v)\cap B_{\frac{1}{\delta_0}}(0)=\pi(X')\cap
B_{\frac{1}{\delta_0}}(0),$$  for some  $v\in B_{\delta_0}(0)$.
Let $0<\varepsilon< \delta_0$. By uniform continuity of $\pi$,
there exists $0<\delta<\frac{1}{s_0}$ such that if  $X$ and $X'$
in $\Omega_{X_1}$ verify $X\cap B_{\frac{1}{\delta}}(0)=X'\cap
B_{\frac{1}{\delta}}(0)$ then
\begin{equation}
\label{continuidad} (\pi(X)-v)\cap
B_{\frac{1}{\varepsilon}}(0)=\pi(X')\cap
B_{\frac{1}{\varepsilon}}(0),
\end{equation}
for some  $v\in B_{\varepsilon}(0)$. Now fix $R\geq
R_{\varepsilon}=\frac{1}{\delta} -s_0$, and  let $X$ and $X'$ be
two Delone sets  in $\Omega_{X_1}$  verifying
\begin{equation}
\label{motivo2} X\cap B_{R+s_0}(0)=X'\cap B_{R+s_0}(0).
\end{equation}
Observe that $X$  and $X'$ satisfy (\ref{continuidad}), and
$(X-a)\cap B_{s_0}(0)=(X'-a)\cap B_{s_0}(0)$, for every $a$ in
$B_{R}(0)$. The choice of $s_0$ ensures  that
\begin{equation}
\label{motivo1} (\pi(X)-a-t(a))\cap
B_{\frac{1}{\delta_0}}(0)=(\pi(X')-a)\cap
B_{\frac{1}{\delta_0}}(0),
\end{equation}
for some $t(a)\in B_{\delta_0}(0)$. Let us prove the map $a\to
t(a)$ is locally constant. For $a\in B_R(0)$, let
$0<s_a<\frac{1}{\delta_0}-R_0$ be such that $B_{s_a}(a)\subseteq
B_R(0)$. Every $a'\in B_{s_a}(0)$ verifies
$B_{\frac{1}{\delta_0}-\|a'\|}(-a') \subset
B_{\frac{1}{\delta_0}}(0)$. Let $a'\in B_{s_a}(0)$. This inclusion
and (\ref{motivo1}) imply
\begin{equation}
\label{new1} (\pi(X)-a-a'-t(a))\cap
B_{\frac{1}{\delta_0}-\|a'\|}(-a')=(\pi(X')-a-a')\cap
B_{\frac{1}{\delta_0}-\|a'\|}(-a').
\end{equation}
On the other hand, from the definition of the map $a\to t(a)$ we
deduce
\begin{equation*}
(\pi(X)-a-a'-t(a+a'))\cap B_{\frac{1}{\delta_0}}(0)=
(\pi(X')-a-a')\cap B_{\frac{1}{\delta_0}}(0),
\end{equation*}
which implies
\begin{equation}
\label{new2} (\pi(X)-a-a'-t(a+a'))\cap
B_{\frac{1}{\delta_0}-\|a'\|}(-a')= (\pi(X')-a-a')\cap
B_{\frac{1}{\delta_0} -\|a'\|}(-a').
\end{equation}
Since $\|t(a)-t(a+a')\| \leq \frac{r_{0}}{2}$, from equations
(\ref{new1}), (\ref{new2}) and the remark of the beginning of the
proof we conclude  $t(a)=t(a+a')$ for every $a' \in B_{s}(0)$.
Therefore the map $a\mapsto t(a)$ is constant on $B_{s_a}(a)$.

Furthermore, due to $\delta_0>\varepsilon$ and  (\ref{motivo2}),
Equation (\ref{continuidad}) implies there exists $v\in
B_{\varepsilon}(0)$ such that
\begin{equation}
\label{motivo3} (\pi(X)-v)\cap
B_{\frac{1}{\delta_0}}(0)=\pi(X')\cap B_{\frac{1}{\delta_0}}(0).
\end{equation}

For $a=0$, from (\ref{motivo1}) and (\ref{motivo3}) we have that
$t(0)=v$ or $\|v-t(0)\|\geq \frac{r_0}{2}$. Since $\|t(0)-v\|\leq
\delta_0+\varepsilon <2\delta_0< \frac{r_{0}}{2}$, we conclude
$t(0)=v$ and then $t(a)=v$ for every $a\in B_R(0)$. This property
together with (\ref{motivo1}) and (\ref{motivo3}) imply that
$$(\pi(X)-v)\cap B_{R}(0)=\pi(X')\cap B_{R}(0).$$ This conclude
the proof.
\end{proof}

\section{Preimages of factor maps.}
\label{Preimages_of_factor_maps}

 In the rest of this paper we suppose that all the Delone sets
are of finite type.

A Delone set $X$ is {\it linearly repetitive}  if there exists a
constant $L>0$ such that for every patch $P$ in $X$, any ball of
radius $L\diam(P)$ intersected with $X$ contains a translated
patch of $P$. In this instance we say that $X$ is {\em linearly
repetitive with constant} $L$. Notice the constant $L$ must be
greater or equal than $1$, and if $X$ is linearly repetitive with
constant $L$, then it is linearly repetitive with constant $L'$,
for every $L'>L$. Every Delone set in the orbit closure of a
linearly repetitive Delone set is linearly repetitive with the
same constant. When $X$ is linearly repetitive, we call
$(\Omega_{X},\RR^d)$ a {\it linearly repetitive} Delone system.

The following lemma shows  the factors of linearly repetitive
systems  are also linearly repetitive with a uniform control on
the constants. This was already proven for subshifts in
\cite{Du1}.

\begin{lemma}
\label{LR} Let $X$ be a linearly repetitive Delone set with
constant $L$. If $X'$ is a Delone set such that
$(\Omega_{X'},\RR^d)$ is a topological factor of
$(\Omega_{X},\RR^d)$, then there exists a constant $\tau_{X'}>0$
such that if $P$ is a patch of $X'$ with $\diam(P)\geq \tau_{X'}$,
then for any $y \in \RR^d$, the set $X' \cap B_{5L\diam(P)}(y)$
contains a translated patch  of $P$.
\end{lemma}

\begin{proof}
Let $\pi:\Omega_{X}\to\Omega_{X'}$ be a topological factor, where
$X$ is a $(r_X,R_X)$-linearly repetitive Delone set  with constant
$L$, and $X'$ is a $(r_{X'},R_{X'})$-Delone set. We can assume
that $\pi(X)=X'$. Let $s_0>0$ be the constant of Lemma
\ref{semi-sliding-block-code}. Fix $0<\varepsilon<Ls_{0}$ and
consider $R_{\varepsilon}>0$  as in Lemma
\ref{semi-sliding-block-code}. We set
$$
\tau_{X'}=\max\{s_0, R_{\varepsilon}, R_{X}, R_{X'}\}.
$$
Let $P$ be a patch in $X'$ with $\diam(P)=D\geq \tau_{X'}$, and
let $v\in P \subset X'$. Let $Q=(X-v)\cap B_{D+s_0}(0)$. Since
$\diam(Q)\leq 2(D+s_0)$, for every $y\in \RR^d$ there exists $w\in
B_{2L(D+s_0)}(y)$ such that $(X-w)\cap B_{D+s_0}(0)=Q$. Then, from
Lemma \ref{semi-sliding-block-code} there exists $t\in
B_{\varepsilon}(0)$ such that
$$(X'-v )\cap B_{D}(0)=(X'-w-t)\cap B_{D}(0).$$
Since $(X'-v )\cap B_{D}(0)$ contains a translated of $P$, this
shows that every ball of radius $2L(D+s_0)+\varepsilon\leq 5LD$ in
$X'$ contains a translated  of $P$ as sub-patch.
\end{proof}


The next Lemma follows the same lines of Lemma 2.4 in \cite{So2}.
We show the set of occurrences of a $R$-patch of a linearly
repetitive Delone set and its factors is uniformly discrete with a
constant depending linearly on $R$.


\begin{lemma}
\label{constantequeminora} Let $X$ be a non periodic linearly
repetitive Delone set with constant $L$, and let $X'$ be  a non
periodic Delone set such that $(\Omega_{X'},\RR^d)$ is a
topological factor of $(\Omega_{X},\RR^d)$. There exists a
constant $M_{X'}>0$ such that for every $R\geq M_{X'}$ and for
every $R$-patch $P$ of $X'$,    if $x \in \RR^d\setminus\{0\}$ is
a return vector of $P$, then  $\| x\| \geq R/(11 L)$.
\end{lemma}

\begin{proof}
Let  $R'>0$ be a real such that any patch of the kind $X' \cap
B_{R'}(y)$, with $y \in\RR^d$,   has  diameter greater than
$\tau_{X'}$, where $\tau_{X'}$ is the constant given by Lemma
\ref{LR}. Let $M_{X'} = 110 L R'+R' $ and   $P$ be the $R$-patch
$X' \cap B_{R}(v)$ with  $R>M_{X'}$ and  $v \in \RR^d$. Suppose
there exists $x\in \RR^d$, with $0<\|x\|<R/(11L)$, such that $P+x$
is a patch of $X'$. For any $y\in \RR^d$, consider the patches
$$
Q_y=X'\cap B_{R'}(y) \mbox{ and}\ S_y = X'\cap B_{R'+\|x\|}(y).
$$

Since
$$
\tau_{X'}\leq \diam(S_y)\leq 2(R'+\|x\|),
$$
from Lemma \ref{LR}, every ball of radius $10L(R'+\|x\|)$
intersected with $X'$ contains a translated of $S_y$. By  the very
hypothesis, we have
$$
10 L(R'+\|x\|) < 10 LR'+ \frac{10 R}{11} \leq \frac{R}{11} +
\frac{10R}{11} =R.$$

This implies  there exists $w\in \RR^d$ such that $S_{y}+w$ is a
sub-patch of $X'\cap B_{R} (v) =P$. Because $P+x$ is also a patch
of $X'$, we have $Q_{y}+w+x$ is also a patch of $X'$ and a
sub-patch of $S_{y}+w$. Hence $Q_y + w +x=Q_{y+x}+w$ and $$Q_{y}
+x=Q_{y+x}.$$  Since $y$ is arbitrary, we conclude that $X'+x=X'$,
which contradicts the non periodicity of $X'$ if $x\neq 0$.
\end{proof}
We recall the following definition: A factor map
$\pi:(\Omega,\RR^d)\to (\Omega',\RR^d)$ is said to be {\it
finite-to-one} (with constant $D$) if for all $y\in Y$ we have
$|\pi^{-1}(\{y\})|\leq D$.

 The next result is a technical lemma we use in  Proposition \ref{finite-to-one0}
to show that factor maps between  linearly repetitive Delone
systems  are finite-to-one.

\begin{lemma}\label{nbrepreimage}
Let $\pi : (\Omega_{X}, \RR^d) \to (\Omega_{X'}, \RR^d) $ be a
factor map, where   $X$ is a linearly repetitive Delone set with
constant $L$, and $X'$ is a non periodic Delone set. We denote by $s_{0}$ the constant given by Lemma \ref{semi-sliding-block-code}. \\
For every $0<\varepsilon<\frac{s_0}{2}$, there exists a constant
$R_{\pi}$ such that for any $R
> R_{\pi}$ there are at most $n\leq
(55L^2)^d$ patches $P_{1}, \ldots, P_{n}$ satisfying for every
$1\leq i\leq n$ the following conditions:
\begin{itemize}
\item[$i)$] $P_{i}=(X-w_{i})\cap B_{R + s_{0}}(0)$, for some
$w_{i} \in \RR^d$,

 \item[$ii)$] If $X"$ belongs to $\Omega_{X}$ and $X" \cap B_{R+s_{0}}(0) = P_{i}$, then
there exists $v\in B_{\epsilon} (0)$ such that
$$ (\pi(X")-v) \cap B_{R}(0) = \pi(X)\cap B_R(0),$$
\item[$iii)$] The patch $  (X-w_{i})\cap B_{R +
s_{0}-2\epsilon}(0)$ is not a sub-patch of $P_{j}$, for every $1
\leq j \leq n$, $j \neq i$.
\end{itemize}

\end{lemma}
\begin{proof}
Let $0<\varepsilon<\frac{s_0}{2}$,  $R_{\pi} = \max \{s_{0},
M_{X'}, R_{\epsilon} \}$ and $R> R_{\pi}$, where $M_{X'}$ is the
constant given by Lemma \ref{constantequeminora} and
$R_{\epsilon}$ by Lemma \ref{semi-sliding-block-code}.  Let
$P_{1}, \ldots, P_{n}$ be $n$ patches of $X$ satisfying the
conditions $i),\ ii),\ iii)$.

Let $1 \leq i \leq n$. We have
$$\diam(P_{i}) \leq 2(R + s_{0}) \leq 4R.$$
Linear repetitivity implies  there exists $v_{i} \in B_{4LR}(0)$
such that $$(X-v_{i}) \bigcap B_{R+s_{0}}(0) = P_{i}.$$ Then by
$ii)$, there is $u_{i} \in B_{\epsilon}(0)$ satisfying
$$
Q = (\pi(X-v_{i})+u_{i}) \cap B_{R}(0) = (\pi(X) -v_{i}+u_{i})
\cap B_{R}(0),
$$
where $Q=\pi(X)\cap B_R(0)$ (observe that $Q$ does not depend on
$i$). This means the set $Q+v_{i}-u_{i}$ is a patch of $\pi(X)$.
As $\{v_{i}-u_{i},\ 1 \leq i \leq n \}$ is included in $B_{4LR+
\epsilon}(0)$ and  $R
>M_{X'}$,   Lemma \ref{constantequeminora}  implies  the
number of elements in $\{v_{i}-u_{i},\ 1 \leq i \leq n \}$ is
bounded by
$$
\frac{\vol( B_{4LR+ \epsilon}(0))}{\vol\left(
B_{\frac{R}{11L}}(0)\right)} \leq (55L^2)^d.
$$

If $n$ is  greater than $(55L^2)^d$, then there exist $i\neq j$
such that $ v_{i}-u_{i} = v_{j}-u_{j}$, and $\vert \vert v_{i}
-v_{j}\vert \vert < 2\epsilon$. This implies  the patch
$(X-v_{i})\cap B_{R+s_{0}-2\epsilon}(0)$ is included in the patch
$(X-v_{j})\cap B_{R+s_{0}}(0) =P_{j} $, which contradicts the
condition $iii)$.
\end{proof}

The next result was proven in \cite{Du1} for subshifts. We use it
with Proposition \ref{igualdad} to conclude the proof of the main
theorem.

\begin{prop}
\label{finite-to-one0} Let $X$ be a linearly repetitive Delone set
with constant $L$. If $\pi:(\Omega_{X},\RR^d)\to
(\Omega_{X'},\RR^d)$ is a factor map such that $X'$ is a non
periodic Delone set, then $\pi$ is finite-to-one with constant
$(55L^2)^d$.
\end{prop}
\begin{proof}
Let $X'_{0} \in \Omega_{X'}$. Suppose there exist $n> (55L^2)^d$
elements $X_{1}, \ldots, X_{n}$ of $\Omega_{X}$,  such that $\pi
(X_{i}) = X'_{0}$, for each $1 \leq i \leq n$. Since they are all
different, there exists $R_{0}>0$ such that for any $R \geq
R_{0}$, the patches
$X_{i}\cap B_{R}(0)$ are pairwise distinct. \\
Let $0<\varepsilon< \frac{s_0}{2}$ and  $R_{\pi}$ be  the constant
given by Lemma \ref{nbrepreimage}. Lemma
\ref{semi-sliding-block-code} ensures that for any  $Y \in
\Omega_{X}$ verifying $Y \cap B_{R}(0) = X_{i}\cap B_{R}(0)$, with
$1\leq i\leq n$    and $R> \max \{R_{0}, R_{\epsilon} +s_{0},
R_{\pi}+s_{0}\}$, there exists $v \in B_{\epsilon}(0)$ such that
$(\pi(Y)-v) \cap B_{R-s_{0}}(0) = X'_{0} \cap B_{R-s_{0}}(0)$.
This means the patches $X_{1}\cap B_{R}(0), \cdots,X_{n}\cap
B_{R}(0)$ satisfy conditions $i)$ and $ii)$ of Lemma
\ref{nbrepreimage}. Then we deduce there exist different $i(R)$
and $j(R)$ in $\{1, \ldots,n \}$ such that the patch $X_{i(R)}\cap
B_{R-2\epsilon}(0)$  is a sub-patch of $X_{j(R)}\cap B_{R}(0)$. In
other words, there exists $v_{R} \in B_{2\epsilon}(0)$ such that
$X_{i(R)} \cap B_{R-2\epsilon}(0) = (X_{j(R)} +v_{R}) \cap B_{R-2\epsilon} (0)$. \\
By the pigeonhole principle, there exist different $i_{0}$ and
$j_{0}$ in $\{1, \ldots, n\}$, and an increasing sequence
$(R_{p})_{p\geq 0}$, tending to $\infty$ with $p$, such that
$i(R_{p}) = i_{0}$ and $j(R_{p})=j_{0}$, for every $p\geq 0$. By
compactness, we can also assume that $(v_{R_{p}})_{p\geq 0}$
converges   to a vector $v$. Thus, for every $p\geq 0$ we get
$$
X_{i_{0}} \cap B_{R_{p}-2\epsilon}(0) = (X_{j_{0}}+v_{R_{p}}) \cap
B_{R_{p}-2\epsilon}(0),
$$
which implies that $X_{i_{0}} = X_{j_{0}} +v$ and $X'_{0} =
\pi(X_{i_{0}}) = \pi(X_{j_{0}} +v) = X'_{0} +v$. Since $X_{i_{0}}
\neq X_{j_{0}}$, the vector $v$ is different from zero, but this
contradicts
the non periodicity of $X'_{0}$. 
\end{proof}

The following proposition is a straightforward generalization of
Lemma 21 in \cite{Du1}.

\begin{prop}
\label{igualdad}  Let  $(\Omega,\RR^d)$ be a minimal Delone system
and $\phi_1:(\Omega,\RR^d)\to(\Omega_1,\RR^d)$,
$\phi_2:(\Omega,\RR^d)\to(\Omega_2,\RR^d)$ be two factor maps.
Suppose that $(\Omega_2,\RR^d)$ is non periodic and  $\phi_1$ is
finite-to-one. If there exist $X,Y\in \Omega$ and $v\in \RR^d$
such that $\phi_1(X)=\phi_1(Y)$ and $\phi_2(X)=\phi_2(Y-v)$, then
$v=0$.
\end{prop}
\begin{proof} There exists a sequence $(v_{i})_{i\in \NN} \subset \RR^d$ such that $\lim_{i\to +\infty}X- v_{i} =Y$. By compactness, we can suppose that
the sequence $(Y-v_{i})_{i\in \NN}$ converges to a point $Y_{2}
\in \Omega$. By continuity, we have $\phi_{1}(Y)=
\phi_{1}(Y_{2})$, and $\phi_{2}(Y) = \phi_{2}(Y_{2})-v$. By
compactness, we can suppose that the sequence of points
$(Y_{2}-v_i)_{i\in \NN} \subset \Omega$ converges to a point
$Y_{3}$. So we have $\phi_{1}(Y_{2}) = \phi_{1}(Y_{3})$ and $
\phi_{2}(Y_{2}) = \phi_{2}(Y_{3})-v$. Hence we construct by
induction a sequence $(Y_{n})_{n\in\NN} \subset \Omega$ such that
$\phi_{1}(Y_{n}) = \phi_{1}(Y_{n+1})$ and $ \phi_{2}(Y_{n}) =
\phi_{2}(Y_{n+1})-v$ for all $n \geq 1$. Since the map $\phi_{1}$
is finite-to-one, there exist $i<j$  such that $Y_{i}=Y_{j}$.
Then, we have
\begin{align*}
\phi_{2}(Y_{i})= \phi_{2}(Y_{i+1})-v = \phi_{2}(Y_{i+2})-2v =
\ldots & =  \phi_{2}(Y_{j})-(j-i)v\\ & =  \phi_{2}(Y_{i})-(j-i)v.
\end{align*}
Since  $(\Omega_{2},\RR^d)$ is non periodic, we conclude   $v=0$.
\end{proof}

\subsection*{Remark} Following the lines of the proof of
Proposition \ref{igualdad}, this result can be generalized to
$\ZZ^d$ or $\RR^d$ actions, more precisely:  Let $G$ be $\RR^d$ or
$\ZZ^d$. Let $(X,G)$ be a minimal dynamical system and
$\phi_1:(X,G)\to(X_1,G)$, $\phi_2:(X,G)\to(X_2,G)$ be two factor
maps. Suppose that $(X_2,G)$ is free and  $\phi_1$ is
finite-to-one. If there exist $x,y\in X$ and $g\in G$ such that
$\phi_1(x)=\phi_1(y)$ and $\phi_2(x)=\phi_2(g.y)$, then $g$ is the
identity in $G$.

\section{Number of factors of linearly repetitive Delone systems.}
\label{Number_of_factors_of_linearly_repetitive_Delone_systems}
Let $X$ be a Delone set of finite type, and  $P=X\cap B_R(x)$ be a
patch of $X$. We define
$$
X_P=\{v\in \RR^d: P+v \mbox{ is a patch of } X \}.
$$
Observe that $0$ always belongs to  $X_{P}$. It is straightforward
to check that $X_P$ is a Delone set when $X$ is repetitive.
Furthermore, $X_{P}$ is a Delone set of finite type because of
$X_P-X_{P} \subset X-X$. Then we define the {\em Vorono\"{\i} cell
of} $P$ associated to $v\in X_P$ as  the Vorono\"{\i} cell of
$v+x\in X_P+x$. That is,
$$
V_{P,v}=\{y\in \RR^d: \|y-(x+v)\|\leq \|y-(x+u)\|, \forall u \in
X_P \}.
$$
Notice  the Vorono\"{\i} cell of $P$ associated to $v\in X_P$ is
the Vorono\"{\i} cell of $v\in X_P$ translated by the vector $x$.


\begin{remark}
\label{important} {\rm It follows from  the definition that a
$(r,R)$-Delone set $X$ satisfies the following: for any $x\in X$,
the diameter of the Vorono\"{\i} cell $V_x$ is smaller or equal to
$2R$ and $B_{\frac{r}{2}}(x)$ is contained in $V_x$. If $X$ is
linearly recurrent with constant $L$, Lemma
\ref{constantequeminora} implies for every sufficiently large $R$
and every  patch $P=X\cap B_R(x)$ of $X$, the collection $X_P$ is
a $(\frac{R}{11L},2LR)$-Delone set. Therefore, in this instance we
have $\diam(V_{P,v})\leq 4LR$ and $B_{\frac{R}{11L}}(x+v)\subseteq
V_{P,v}$, for every $v\in X_P$.}
\end{remark}

In the next lemma, we bound the number of ways we can prolong a
given patch $P$ to a bigger one. More precisely, this gives an
upper bound of the number (up to translation) of $R'$-patches
$X\cap B_{R'}(x)$, such that $X\cap B_R(x)$ is a translated of
$P$.

\begin{lemma}
\label{voronoi_new} Let $X$ be a linearly repetitive Delone set
with constant $L$, and consider $0<R_1<R_2$, with $R_1$
sufficiently large. Then there are at most $n\leq
(44L^2)^d\left(\frac{R_2}{R_1} \right)^d$ patches $P_1,\cdots,
P_n$ of $X$, up to translation, satisfying for every $1\leq i\leq
n$ the following two conditions:
\begin{enumerate}
\item[$i)$] there exists $v_i\in\RR^d$ such that $P_i=X\cap
B_{R_2}(v_i)$.

\item[$ii)$] $(X-v_i)\cap B_{R_1}(0)=(X-v_j)\cap B_{R_1}(0)$, for
every $1\leq j\leq n$.
\end{enumerate}
\end{lemma}
\begin{proof}
Applying Lemma \ref{constantequeminora} to the identity factor map
on $(\Omega_{X}, \RR^d)$, we deduce  there exists $M_{X}
>0$, such that  for every $R\geq M_{X}$ and $x\in \RR^d$, the distance between two different occurrences of $P=X\cap B_R(x)$ is greater or equal to
$R/(11L)$.

Let $M_X\leq R_1<R_2$ and $n\in \NN$. Suppose $P_1,\cdots, P_n$
are patches of $X$ verifying conditions $i)$ and $ii)$, and such
that for every $1\leq i\leq n$,
\begin{itemize}
\item[$iii)$] $P_i$ is not a translated of $P_j$, for every $j\in
\{1,\cdots,n\}\setminus\{i\}$.
\end{itemize}

Condition $i)$ and linear repetitivity of $X$ imply for every
$1\leq i\leq n$, there exists $w_i\in \RR^d$ such that $P_i+w_i$
is a sub-patch of $X\cap B_{2LR_2}(0)$. From condition $ii)$ it
follows that for every $1\leq i\leq n$, the point $v_i+w_i$ is an
occurrence of the patch $X\cap B_{R_1}(v_1)$ in the ball
$B_{2LR_2}(0)$. Finally, by the choice of $R_1$, conditions $ii)$,
$iii)$ and Lemma \ref{constantequeminora},  for every $i$ and $j$
in $\{1,\cdots, n\}$ such that $i\neq j$, we get
$\|v_i+w_i-(v_j+w_j)\|\geq \frac{R_1}{11L},$ which implies
$$
n\leq
\frac{\vol(B_{2LR_2}(0))}{\vol(B_{\frac{R_1}{22L}}(0))}=(44L^2)^d\left(
\frac{R_2}{R_1}\right)^d,
$$
and  achieves the proof.
\end{proof}

The following lemma is certainly well-known, but we did not find
any reference. This shows that a Voronoi cell of a point $x$ in a
$(r,R)$-Delone set $X$ is completely determined by the points in
$X\cap B_{4R}(x)$.

\begin{lemma}
\label{voronoi} Let $X$ be a $(r,R)$-Delone set. Then for every
$x\in X$ it holds
$$
V_x=\{y\in \RR^d: \|x-y\|\leq \|x'-y\|, \mbox{ for every } x'\in
X\cap B_{4R}(x) \}.
$$
\end{lemma}
\begin{proof}
Let $C_x=\{y\in \RR^d: \|x-y\|\leq \|x'-y\|, \mbox{ for every }
x'\in X\cap B_{4R}(x) \}$.

By definition of Vorono\"{\i} cell, the inclusion $V_x\subseteq
C_x$ is direct.

Observe the set $C_x$ is convex because is obtained as
intersection of convex sets. Now, suppose there exists $y\in
C_x\setminus V_x$. Then there exist $x'\in X$, verifying $V_x\cap
V_{x'}\neq \emptyset$, and $z\in ([x,y]\cap V_{x'})\setminus V_x$,
where $[x,y]$ is the segment with extreme points $x$ and $y$.
Since $\|x-x'\|\leq 4R$ and $\|z-x'\|<\|z-x\|$, definition of
$C_x$ implies $z\notin C_x$, which contradicts the convexity of
$C_x$.
\end{proof}

\begin{lemma}
\label{cotavoronoi} Let $X$ be a non periodic linearly repetitive
Delone set with constant $L$. There exists a positive constant
$c(L)$ such that for every  sufficiently large $R$ and every patch
$P=X\cap B_R(x)$, the collection $\{X\cap V_{P,v}: v\in X_P\}$
contains at most
 $c(L)$ elements up to translation.
\end{lemma}
\begin{proof}
Let $R$ be a big enough positive number, in order to apply Lemma
\ref{voronoi_new} to $R_1=R$ and $R_2=8LR$.

Let $x\in \RR^d$, $P=X\cap B_R(x)$ and $v\in X_P$. Since $X_P+x$
is a Delone set with constant of uniform density equal to $2LR$
(see Remark \ref{important}),  Lemma \ref{voronoi}  implies
$V_{P,v}$ is completely determined by the patch $X\cap
B_{8RL}(v+x)$. Furthermore, the Vorono\"{\i} cell $V_{P,v}$ is
contained in the ball $B_{4RL}(v+x)$ (see Remark \ref{important}).
Then it follows there are at most as many Vorono\"{\i} cells of
$P$ and patches of the kind $X\cap V_{P,v}$, up to translation, as
patches $Q$ satisfying the following two conditions: $i)$ there
exists $w\in \RR^d$ such that $Q=X\cap B_{8RL}(w)$  and $ii)$ $w$
is an occurrence of a translated of $P$. These two conditions and
Lemma \ref{voronoi_new} imply there are at most
$$
c(L)\leq (44L^2)^d\left(\frac{8LR}{R} \right)^d=(352L^3)^d
$$
patches of the kind $X\cap V_{P,v}$  up to translation.
\end{proof}


We have already defined the notion of return vector of a patch,
now let us define the notion of return vector of a Vorono\"{\i}
cell of a patch. For a patch $P=X\cap B_R(x)$ of $X$ and $v\in
X_P$, we say that $w\in \RR^d$ is a {\em return vector of}
$V_{P,v}\cap X$ if $(X-w)\cap V_{P,v}= X\cap V_{P,v}$. We set
$$P_{n,w,v} \mbox{ the patch }  (X-w-x-v)\cap B_{L^nR}(0).$$
Notice that $P_{n,w,v}+v+w+x$ is a patch of $X$. When there is no
confusion about $n$ and $v$, we write  $P_w$ instead of
$P_{n,w,v}$.

\begin{lemma}
\label{cotavoronoi3} Let $n\in \NN$ and  $X$ be a  non periodic
linearly repetitive Delone set with constant $L$. For every
sufficiently large $R>0$ and every $R$-patch $P$,  the collection
$\{P_w: w \mbox{ is a return vector of }V_{P,v}\cap X\}$ has at
most $c(n,L)$ elements, for every $v\in X_P$.
\end{lemma}

\begin{proof}
Let $R_1=R$ and $R_2=L^nR$ be sufficiently large positive numbers
in order to apply Lemma \ref{voronoi_new}. Let $P=X\cap B_R(x)$ be
a patch of $X$ and $v\in X_P$. Since $X_P+x$ is a Delone set with
constant of uniform discreteness equal to $\frac{R}{11L}$, the
Vorono\"{\i} cell $V_{P,v}$ contains the ball
$B_{\frac{R}{22}}(v+x)$. This implies for every pair of return
vector $u$ and $w$ of $V_{P,v}$ it holds $P_w\cap
B_{\frac{R}{22}}(0)=P_u\cap B_{\frac{R}{22}}(0)$. Thus, from Lemma
\ref{voronoi_new} it follows there are at most
$$
c(n,L)\leq(44L^2)^d\left( \frac{L^{n}R}{\frac{R}{22L}}\right
)^d=(968L^{n+3})^d
$$
patches of the kind $P_w$.
\end{proof}

Let $n\in\NN$.   We call $M(n,L)$ the number of coverings of a set
with $c(L)c(n,L)$ elements, where $c(L)$ and $c(n,L)$ are the
constants of Lemma \ref{cotavoronoi} and Lemma \ref{cotavoronoi3}
respectively.

\begin{theo}
Let $X$ be a linearly repetitive Delone set with constant $L$.
There are finitely many Delone system factors of $(X,\RR^d)$  up
to conjugacy.
\end{theo}



\begin{proof}
Let $X$ be a non periodic linearly repetitive Delone set with
constant $L>1$. Let $n\in \NN$ be such that
\begin{equation}
\label{n_value} L^n-1-12L-176L^2>1,
\end{equation}
and let $R_1>1$ be a constant such that for every $R\geq R_1$,
Lemma \ref{cotavoronoi} and Lemma \ref{cotavoronoi3} are
applicable to $R$-patches of $X$.

For every $1\leq i\leq M(n,L)+1$, let $X_i$ be a non periodic
Delone set such that there exists a topological factor map
$\pi_i:\Omega_{X}\to \Omega_{X_i}$, and let $X_{0} = X$. Let
$M_{X_i}$ be the constant of  Lemma \ref{constantequeminora}
associated to $X_i$.

Fix $0<\varepsilon<1$. For every $1\leq i\leq M(n,L)+1$, consider
$R_{\varepsilon}^{(i)}$ and $s_0^{(i)}$ the constants of Lemma
\ref{semi-sliding-block-code} associated to $\pi_i$. We define
$$
R_{\varepsilon}= \max_{i}\{R_{\varepsilon}^{(i)}\},\,\, s_0 =
\max_{i}\{s_0^{(i)}\} \mbox{ and }  M = \max_{i}\{M_{X_i}\}.
$$
Observe  in an open ball of radius $r/22L$, there is at most one
return vector of a $r$-patch of $X_i$, with $r\geq M$, for every
$1\leq i\leq M(n,L)+1$.

We take
$$
R>\max\{R_{\varepsilon}, s_0, M+\varepsilon, R_1, 45L \},
$$

Consider the patch $P=B_R(0)\cap X$, and $v_1,\cdots, v_N\in X_P$
such that for every $v\in X_P$, there exist $1\leq i\leq N$ and
$u\in \RR^d$ verifying $V_{P,v}\cap X=(V_{P,v_i}\cap X)+u$.
Roughly speaking, every set of the kind $V_{P,v}\cap X$  is a
translated of some set $V_{P,v_i}\cap X$. Since  $R>R_1$, Lemma
\ref{cotavoronoi} ensures $N\leq c(L)$.

For every $1\leq j\leq N$, let $w_{j,1},\cdots, w_{j,m_j}$ be
return vectors of $V_{P,v_j}\cap X$, chosen in order that for
every return vector $w$ of $V_{P,v_j}\cap X$, there exists $1\leq
i\leq m_j$ such that $P_w$ is equal to $P_{w_{j,i}}$. Since
$R>R_1$, Lemma \ref{cotavoronoi3} implies that $m_j\leq c(n,L)$,
for every $1\leq j\leq N$. Therefore, the collection
$$
\F=\{P_{w_{j,l}}: 1\leq l\leq m_j,\,\, 1\leq j\leq N\}
$$
contains at most $c(L)c(n,L)$ elements.

Let $R'$ be the constant given by
$$
R'=(L^{n}-1)R-\varepsilon-4LR.
$$
The choice of $n$ ensures that $R'>0$.


For every $1\leq i\leq M(n,L)+1$, we define the following relation
on $\F$:
$$
\begin{array} {c}
P_{w_{j,l}} \R_i P_{w_{k,m}}\\
 \Updownarrow\\
\mbox{ for every } X', X''\in \Omega_{X} \mbox{ such that }\\ X'\cap B_{L^nR}(0)=P_{w_{j,l}} \mbox{ and } X''\cap B_{L^nR}(0)=P_{w_{k,m}},\\
\mbox{ there exist } v\in B_{2\varepsilon}(0) \mbox{ and } w\in
B_{4LR}(0) \mbox{ such that }\\
\pi_i(X')\cap B_{R'}(0)=(\pi_i(X'')+v+w)\cap B_{R'}(0).
\end{array}
$$

Since $L^nR-s_0\geq (L^n-1)R\geq R>R_{\varepsilon}$, from Lemma
\ref{semi-sliding-block-code} it follows this relation is
reflexive, so non empty. Since the cardinal of $\F$ is bounded by
$c(L)c(n,L)$, there are at most $M(n,L)$ different relations of
this kind. So, there exist $1\leq i<j<M(n,L)+1$ such that
$\R_i=\R_j$.

In the sequel, we will prove that $(\Omega_{X_i},\RR^d)$ and
$(\Omega_{X_j},\RR^d)$ are conjugate. For that, it is sufficient
to show  that if $Y, Z\in \Omega_{X}$ are such that
$\pi_i(Y)=\pi_i(Z)$ then $\pi_j(Y)=\pi_j(Z)$.

Let $Y$ and $Z$ be two Delone sets in $\Omega_{X}$ such that
$\pi_i(Y)=\pi_i(Z)$. Without lost of generality, we can suppose
that $0$ is an occurrence of  $P$ in $Y$ and in $Z-u_0$, where
$u_0$ is some point in $B_{4LR}(0)$. The  patches of $Y$ and $Z$
are translated of the patches of $X$. This implies there exist
$1\leq q_0, r_0\leq N$ such that
$$
Y\cap B_{L^nR}(0)=P_{w_{q_0,l_0}} \mbox{ and } (Z-u_0)\cap
B_{L^nR}(0)=P_{w_{r_0,k_0}},
$$
for some $1\leq l_0\leq m_{q_0}$ and $1\leq k_0\leq m_{r_0}$

\medskip

{\it \underline{Claim 1}:} $P_{w_{q_0,l_0}}\R_i P_{w_{r_0,k_0}}$.

\medskip

{\it Proof of Claim 1:}  Let $X'$ and $X''$ be two Delone sets in
$\Omega_{X}$ such that $X'\cap B_{L^nR}(0)=P_{w_{q_0,l_0}}$ and
$X''\cap B_{L^nR}(0)=P_{w_{r_0,l_0}}$. Since $R\geq s_0$, $R\geq
R_{\varepsilon}$ and
$$
X'\cap B_{L^nR}(0)=Y\cap B_{L^nR}(0), \,\, X''\cap
B_{L^nR}(0)=(Z-u_0)\cap B_{L^nR}(0),
$$
By the choice of $n$ and $R$, Lemma \ref{semi-sliding-block-code}
implies there exits $z_1$ and $z_2$ in $B_{\varepsilon}(0)$ such
that
\begin{align*}
(\pi_i(X')+z_1)\cap B_{(L^{n}-1)R}(0)& =  \pi_i(Y)\cap
B_{(L^{n}-1)R}(0), \mbox{ and }\\
(\pi_i(X'')+z_2)\cap B_{(L^{n}-1)R}(0) & = \pi_i(Z-u_0)\cap
B_{(L^{n}-1)R}(0).
\end{align*}
Then we get
\begin{align*}
&(\pi_i(X'')+z_2+u_0)\cap B_{(L^{n}-1)R-4LR}(0)\\
= &
\pi_i(Z)\cap B_{(L^{n}-1)R-4LR}(0)\\
  = & \pi_i(Y)\cap B_{(L^{n}-1)R-4LR}(0)\\
  = & (\pi_i(X')+z_1)\cap B_{(L^{n}-1)R-4LR}(0).
\end{align*}
Therefore
$$
(\pi_i(X'')+z_2+u_0-z_1)\cap B_{(L^{n}-1)R-4LR-\varepsilon}(0)=
\pi_i(X') \cap B_{(L^{n}-1)R-4LR-\varepsilon}(0),
$$
which implies that $P_{w_{q_0,l_0}}\R_i P_{w_{r_0,k_0}}$.

\medskip

Since $\R_i=\R_j$, from Claim 1 we get $P_{w_{q_0,l_0}}\R_j
P_{w_{r_0,k_0}}$.

Let $s$ be any other occurrence of $P$ in $Y$. Repeating the same
argument for $Y+s$ and $Z+s$, we deduce there exist $u_s\in
B_{4LR}(0)$ and $1\leq q_s, r_s\leq N$ such that
$$
(Y+s)\cap B_{L^nR}(0)=P_{w_{q_s,l_s}} \mbox{ and } (Z-u_s)\cap
B_{L^nR}(0)=P_{w_{r_s,k_s}},
$$
for some $1\leq l_s\leq m_{q_s}$ and $1\leq k_s\leq m_{r_s}$. Then
from Claim 1 we get $P_{w_{q_s,l_s}}\R_jP_{w_{r_s,k_s}}$. This
implies  there exist $t_s\in B_{2\varepsilon}(0)$ and $w_s\in
B_{4LR}(0)$ such that
$$
\pi_j(Y+s)\cap B_{R'}(0)=(\pi_j(Z+s-u_s)+t_s+w_s)\cap B_{R'}(0).
$$

\medskip

{\it \underline{Claim 2}:} The vector $w_s-u_s+t_s$ does not
depend on $s$, i.e, there exists  $y\in\RR^d$ such that
$w_s-u_s+t_s=y$ for every occurrence $s$ of $P$ in  $Y$.

\medskip

{\it Proof of Claim 2:} Let $s_1$ and $s_2$ be two occurrences of
$P$ in $Y$ such that the Vorono\"{\i} cells of $s_1$ and $s_2$,
with respect to set of occurrences of $P$ in $Y$,  have common
points in their borders. Since the diameter of these Vorono\"{\i}
cells is smaller or equal to $4RL$ (see remark \ref{important}),
we get $\|s_1-s_2\|\leq 8LR$. Then Then
\begin{eqnarray*}
& &(\pi_j(Z) +s_1+(s_2-s_1)-u_{s_1}+t_{s_1}+w_{s_1})\cap B_{R'-
8LR}(0)\\
&=& (\pi_j(Y)+s_1+(s_2-s_1))\cap B_{R'- 8LR}(0)\\
&=& (\pi_j(Z)+s_2-u_{s_2}+t_{s_2}+w_{s_2})\cap B_{R'- 8LR}(0).
\end{eqnarray*}
This implies
$(-u_{s_1}+t_{s_1}+w_{s_1})-(-u_{s_2}+t_{s_2}+w_{s_2})$ is a
return vector of a $(R'- 8LR)$-patch of $\pi_j(Z)+s_2$. Since
$$
R'-8LR=R(L^n-1-12L)-\varepsilon\geq R-\varepsilon>M,
$$
Lemma \ref{constantequeminora} implies the non zero vectors of the
$(R'-8LR)$-patches of $\pi_j(Z)+s_2$ have norm greater or equal to
$(R'-8LR)/11L$. Thus, due to
$$
\|-u_{s_1}+t_{s_1}+w_{s_1}-(-u_{s_2}+t_{s_2}+w_{s_2})\|\leq
16LR+4\varepsilon,
$$
and
\begin{eqnarray*}
11(16LR+4\varepsilon) &=& 176L^2R+44L\varepsilon\\
                      &< &(L^n-1-12L-1)R+44L\varepsilon\\
                      & = & R'-8LR+\varepsilon-R+44L\varepsilon\\
                      & < & R'-8LR+L-R+44L<R'-8LR,
\end{eqnarray*}
we deduce  $-u_{s_1}+t_{s_1}+w_{s_1}=-u_{s_2}+t_{s_2}+w_{s_2}$,
which  shows Claim 2.

\medskip

From Claim 2 we get there exists $y\in \RR^d$ such that for every
occurrence $s$ of $P$ in $Y$,
\begin{align*}
\pi_j(Y+s)\cap B_{R'}(0) & =(\pi_j(Z+s)+y)\cap B_{R'}(0), \mbox{
and then }\\
\pi_j(Y)\cap B_{R'}(s) & =(\pi_j(Z)+y)\cap B_{R'}(s).
\end{align*}
From Remark \ref{important}, the diameter of the Vorono\"{\i}
cells of $P$ is less than $4LR$, which is less than $R'$. Hence,
$$
\pi_j(Y)=\pi_j(Z)+y.
$$
We conclude with Lemma \ref{finite-to-one0} and Proposition
\ref{igualdad}.
\end{proof}

\subsection*{Acknowledgments}
MIC is grateful for the grant  FONDECYT de Iniciaci{\'o}n
11060002. She also thanks the Laboratoire Ami\'enois de
Ma\-th\'e\-ma\-ti\-ques Fondamentales et Appliqu\'ees, CNRS-UMR
6140, Universit\'{e} de Picardie Jules Verne, where part of this
work was done. FD would like to thank the hospitality of
Departamento de Matem\'atica y Ciencia de la Computaci\'on de la
Universidad de Santiago de Chile. MIC and FD were partially
supported by Nucleus Millenius P04-069-F. The three authors thank
the financial support of program CNRS/CONICYT 2008 N$^o$ 21202.

\end{document}